\documentclass[a4paper, 11pt]{amsart}

\usepackage[T1]{fontenc}
\usepackage[utf8]{inputenc}

\usepackage{amssymb}
\usepackage{mathrsfs}

\allowdisplaybreaks[2]

\newtheorem{theorem}{Theorem}[section]
\newtheorem{proposition}[theorem]{Proposition}
\newtheorem{lemma}[theorem]{Lemma}
\newtheorem{corollary}[theorem]{Corollary}
\theoremstyle{definition}
\newtheorem{remark}[theorem]{Remark}
\newtheorem{definition}[theorem]{Definition}
\newtheorem{example}[theorem]{Example}
\newtheorem{examples}[theorem]{Examples}

\numberwithin{equation}{section}

\newcommand{\N}{\mathbb N}
\newcommand{\No}{\N_0}
\newcommand{\Z}{\mathbb Z}

\newcommand{\M}{I}
\newcommand{\m}{i}

\newcommand{\BBB}[1]{\mathscr{B}_{#1}}
\newcommand{\BBBi}[1]{\mathscr{B}'_{#1}}
\newcommand{\CCC}[3]{\mathscr{C}_{#1,#2,#3}}
\newcommand{\CCCt}[2]{\mathscr{C}_{#1,#2}}
\newcommand{\DDD}[2]{\mathscr{D}_{#1,#2}}
\newcommand{\bbb}{\textbf{b}}
\newcommand{\ccc}{\textbf{c}}
\newcommand{\ddd}{\textbf{d}}

\newcommand{\bbbi}{\bbb'}

\newcommand{\Y}{V}
\newcommand{\Yo}{U}

\newcommand{\y}{v}
\newcommand{\yo}{u}

\newcommand{\pcp}[3]{#1 \langle #2 \rangle \amalg #1 [#3]}

\newcommand{\h}[1]{\pi_{#1}}
\newcommand{\ho}{\pi_{0}}
\newcommand{\hi}{\pi_{1}}

\newcommand{\e}[1]{\textbf{e}_{#1}}
\newcommand{\lm}[1]{\ell_{#1}}

\newcommand{\Ao}{A_\y}

\newcommand{\skal}[3]{\alpha_{#1, #3}^{#2}}
\newcommand{\skalt}[2]{\alpha_{#2}^{#1}}
\newcommand{\nabs}[2]{M_{#2}^{#1}}

\newcommand{\lam}[3]{\lambda_{#1, #3}^{#2}}
\newcommand{\lama}[3]{\lambda_{#1, #3}^{#2}}
\newcommand{\lamt}[2]{\lambda_{#1}^{#2}}
\newcommand{\mm}[3]{\mu_{#1, #3}^{#2}}
\newcommand{\mmt}[2]{\mu_{#2}^{#1}}

\newcommand{\pol}[4]{h_{#1, #3, #4}^{#2}}

\newcommand{\Ev}[2]{{\rm Ev}_{#1 ; #2}}

\title[Multilinear polynomials]{Multilinear polynomials are surjective on algebras with surjective inner derivations}

\author{Daniel Vitas}
\address{Department of Mathematics, Faculty of Mathematics and Physics,  University of Ljubljana, Slovenia}
\email{daniel.vitas@student.fmf.uni-lj.si}

\thanks{\emph{Mathematics Subject Classification} (2020). 16R99,  16W25}
\keywords{Multilinear noncommutative polynomial, surjective inner derivation,  L'vov-Kaplansky conjecture.}

\begin{document}

\begin{abstract} Let $f(X_1,\dots,X_n)$ be a nonzero multilinear noncommutative polynomial. If $A$ is a unital algebra with a surjective inner derivation, then every element 
in $A$ can be written as $f(a_1,\dots,a_n)$ for some $a_i\in A$.
\end{abstract}

\maketitle

\section{Introduction}

Let $F$ be a field. By $F\langle X_1,\dots,X_n\rangle$ we denote the free algebra in  $X_i$ over $F$; its elements are called   {\em noncommutative polynomials}. 
Let $A$ be any algebra over $F$ and let $f=f(X_1,\dots,X_n)$ be any noncommutative polynomial.  The set
$$f(A)=\{f(a_1,\dots,a_n)\,|\, a_1,\dots,a_n\in A\}$$
 is called the {\em image of $f$}. If $f(A)=A$, then $f$ is said to be {\em surjective on $A$}.

We say that $f$ is {\em multilinear} if it is of the form \[
  f(X_1, \ldots, X_n) = \sum_{\sigma \in S_n} \lambda_{\sigma} X_{\sigma(1)} \ldots X_{\sigma(n)}
\]
for some $\lambda_{\sigma}\in F$. The {\em L'vov-Kaplansky conjecture} states that if $F$ is an infinite field and $A=M_n(F)$ is the algebra of $n\times n$ matrices over $F$, 
then $f(A)$ is a vector space (as a matter of fact, it can only be one of the following four vector spaces: $\{0\}$, the space of scalar matrices $F1$, the space of traceless matrices $[A,A]$, and $A$). In \cite{KBMR},
Kanel-Belov, Malev, and Rowen  confirmed this conjecture for the case where $n=2$ and $F$ is quadratically closed. Since then, there has been a lot of effort by several authors
 to solve the conjecture for larger $n$; however, even the $n=3$ case is currently  only partially solved. Various variations of the problem  have also been proposed and studied. We refer the reader to the recent survey article \cite{RYMKB} for a thorough account of this line of investigation.

The initial idea, from which this paper arose, was to consider 
 the infinite-dimension\-al version of the L'vov-Kaplansky conjecture.  More specifically, if
$V$ is an infinite-dimensional vector space over $F$ and $A={\rm End}_F(V)$, then one can ask whether $f(A)$ is a vector space  for every multilinear polynomial $f$.
We will see that this is indeed true. In fact, if $f\ne 0$ then $f(A)$ is simply equal to $A$, i.e., $f$ is surjective on $A$. We remark that the assumption that $f$ is multilinear is indeed necessary---see \cite[Example 3.16]{B}. 

We will actually prove  that this result, i.e., $f(A)=A$ with $f$  multilinear and nonzero,  holds for 
  a considerably larger class of algebras, which we now introduce.
Recall that a map  of the form $x\mapsto [\y,x]$, where $\y$ is a fixed element in $A$, is called an {\em inner derivation} of $A$ (here,
$[\y,x]$ stands for $\y x-x\y$). We will 
be interested in  algebras $A$ that have  a {\em surjective inner derivation}. 
That is, there exists a $\y \in A$ 
such that, for every $y \in A$, there is an $x\in A$ satisfying $[\y,x] = y$.

As inner derivations  have nontrivial kernels, no finite-dimensional algebra has a surjective inner derivation.
Moreover,  the same is true for any PI-algebra $A$, since $[A,A]$ is always a proper subset of $A$ \cite{KB}.
 Nevertheless, the class of algebras  with surjective inner derivations
 is fairly large, as we will now show. 

\begin{examples}
1. Let $V$ be an infinite-dimensional vector space over a division algebra $D$ (over any field $F$), and let $A={\rm End}_D(V)$. Let $\{ e_{\m, n} \mid (\m,n)\in \M \times \N \}$ be a basis of
$V$ (here we used the standard fact that every infinite set $\M$ has the same cardinality as $\M \times \N$)
 and let $\y \in A$ be given by $\y(e_{\m,n})= e_{\m, n+1}$.
Take $y\in A$. Then $x\in A$  defined by $x(e_{\m,1})=0$ and
$$x(e_{\m,n}) = -\y^{n-2}y(e_{\m,1}) - \y^{n-3}y(e_{\m,2}) -\dots - \y y(e_{\m,n-2}) - y(e_{\m,n-1}),\,\,\,n\ge 2,$$
satisfies $[\y,x] =y$. Thus,  $A$ is an algebra with a surjective inner derivation.

 2. Let  char$(F)=0$.
Suppose an algebra $A$ is generated by a pair of elements $\y,w$ satisfying $[\y,w]=1$ together with the elements that commute with both $\y$ and $w$.
It is easy to see that  every element in $A$
 is then a linear combination of elements of the form $\y^k w^\ell t$ with $k,\ell\ge 0$ and $[t,\y]=[t,w]=0$. Since $[\y,\frac{1}{\ell+1}\y ^{k}w^{\ell+1} t]= \y^k w^\ell t$, it follows that $ [\y,A]=A$, so $A$ is an algebra with a surjective inner derivation. Important concrete examples of such algebras $A$ are   Weyl algebras $A_n(F)$, $n\in \mathbb N$.

3. From the paper by Cohn \cite{C} it is evident that there exist division algebras with surjective inner derivations. 

4. The direct product of any family, finite or infinite, of algebras with surjective inner derivations  is an algebra with a surjective inner derivation.

5. A homomorphic image of an algebra  with a surjective inner derivation is again an algebra with a surjective inner derivation.

6. If $A$ is an algebra with a surjective inner derivation and $B$ is any unital algebra, then $A\otimes B$ is an algebra with a surjective inner derivation. Indeed, if $x\mapsto [\y,x]$ is a surjective inner derivation of $A$,
then $x\mapsto [\y\otimes 1,x]$ is a surjective inner derivation of $A\otimes B$.
\end{examples}
We will establish the following theorem.

\begin{theorem}\label{t}
If $A$ is a unital algebra with a surjective inner derivation, then  every nonzero multilinear polynomial is surjective on $A$. 
\end{theorem}

Let us point out the two most prominent examples to which this theorem applies.

\begin{corollary}\label{cend}
Let $V$ be an infinite-dimensional vector space over a division algebra $D$ (over any field $F$). Then 
 every nonzero multilinear polynomial is surjective on the algebra ${\rm End}_D(V)$.
\end{corollary}

\begin{corollary}\label{cweyl}
 Every nonzero multilinear polynomial is surjective  on the $n$th Weyl algebra $A_n(F)$  (here, $F$ is any field with characteristic $0$). 
\end{corollary}

We will actually prove a more general result, Theorem \ref{mt}, which involves somewhat more general polynomials
(which we call admissible partially commutative polynomials) and requires less than the surjectivity of a derivation. What follows is devoted to its proof.
More precisely, in Section \ref{s2} we provide an appropriate setting for our problem, in Section \ref{s3} we consider a special system of linear equations that occurs in the proof, and in Section \ref{s4} we give the proof of the main result.

\section{Admissible partially commutative polynomials} \label{s2}

Let $F$ be a field. By
\[
 \pcp{F}{X_1,\ldots,X_n}{\Yo,\Y}
\]
we denote the coproduct  (see, e.g., \cite[Section 1.4]{Rings}) of the free algebra $F \langle X_1, \ldots, X_n \rangle$, i.e., the algebra of polynomials in noncommuting  variables $X_i$, and 
 $ F[\Yo, \Y]$, the algebra of polynomials in two commuting  variables $\Yo$ and $\Y$.
We will refer to the elements of this coproduct algebra as {\em partially commutative polynomials}. We may think of them as polynomials in $X_i,\Yo,\Y$ where the variables $\Yo$ and $\Y$ commute among themselves, but do not commute with the variables $X_i$.

Let $A$ be a unital algebra over $F$ and let $\y$ be a fixed element in $A$. In Section \ref{s4} we will impose some conditions on $\y$, but at this stage it can be any element. 
By $F[\y]$ we denote the (unital) subalgebra of $A$ generated by $\y$. Take 
any $x_1,\dots,x_n\in A$ and $\yo \in F[\y]$, and let 
\[
  \Ev{x_1, \ldots, x_n}{\yo}: \pcp{F}{X_1,\ldots,X_n}{\Yo,\Y} \rightarrow A \text{}
\]
be the algebra homomorphism sending $X_i$ to $x_i$, $\Yo$ to $\yo$, and $\Y$ to $\y$. Note that $\Ev{x_1, \ldots, x_n}{\yo}$ extends the
usual evaluation homomorphisms on $F \langle X_1, \ldots, X_n \rangle$ and $F[\Yo, \Y]$, respectively, so its existence follows from the standard properties of
the coproduct. We define the {\em image of a partially commutative polynomial} $$f\in \pcp{F}{X_1,\ldots,X_n}{\Yo,\Y}$$ as
\[
  f(A;\y) = \left\{ \Ev{x_1, \ldots, x_n}{\yo}(f) \,\, \middle\vert \,\, x_1,\ldots, x_n \in A,~\yo \in F[\y] \right\} \text{.}
\]
If $f$ does not involve the variable $\Yo$, we will  write $$\Ev{x_1, \ldots, x_n}{}(f)$$
 instead of $\Ev{x_1, \ldots, x_n}{\yo}(f)$.
Of course, if $f$ is a noncommutative polynomial, i.e,  an element of  the subalgebra $F \langle X_1, \ldots, X_n\rangle$ of $ F \langle X_1, \ldots, X_n \rangle \amalg F[\Yo, \Y]$, then
$f(A)=f(A;\y)$. (We actually will not deal with  images of partially commutative polynomials until Section \ref{s4}, but 
giving this definition at this early stage  may help the reader to understand the context.)

The reason for introducing partially commutative polynomials is the method of our proof.  In order to get closer to ``pure'' noncommutative polynomials (i.e., elements  of
$ F \langle X_1, \ldots, X_n\rangle$) in which we are primarily interested, we will, in
 the course of proof, often substitute  $\Y^k$ for $\Yo$. Let 
\[
  \h{k}: \pcp{F}{X_1,\ldots,X_n}{\Yo,\Y} \rightarrow  \pcp{F}{X_1,\ldots,X_n}{\Yo, \Y}
\]
be the  homomorphism that fixes each $X_i$ and $\Y$, and sends $\Yo$ to  $\Y^k$.
A routine proof shows that
\begin{equation}\label{gg}
   \h{k}(f)(A;\y) \subseteq f(A; \y) 
\end{equation}
for every  $f \in \pcp{F}{X_1,\ldots,X_n}{\Yo,\Y}$ and every $k\in \mathbb N_0 =\N\cup  \{0\}$.

We need some more notation. First of all, we will write
\[
  [x_1, x_2, \ldots, x_n] = [x_1, [x_2, \ldots, x_n]] 
\]
and 
$$
[x_1,x_2]_k =   [\underbrace{x_1, \ldots, x_1}_k, x_2] \text{.}
$$

For any $r \in \No$, we define
\[
  \BBB{r} = \left\{ \bbb = (\bbb_1,\ldots, \bbb_n) \in \No^n \,\, \middle\vert \,\, \sum_{i=1}^n \bbb_i = r \right\} \text{,}
\]
and for any 
$\bbb \in \BBB{r}$ we define
\[
  X_j^{\bbb} = [\underbrace{\Y, \ldots, \Y}_{\bbb_j}, X_j] = [\Y,X_j]_{\bbb_j}\text{;}
\]
if $\bbb_j=0$,  it should be understood that $  X_j^{\bbb} = X_j$.
For $i \in \{1, \ldots, n \}$, we define
\[
  X_j^{\bbb, i} = \begin{cases} [\Yo, X_j^{\bbb}],~j = i \\ X_j^{\bbb},~j \neq i \end{cases} \text{.}
\]
We extend both definitions by setting
\begin{align*}
   \left( X_{j_1} X_{j_2} \ldots X_{j_k}\right)^{\bbb} &= X_{j_1}^{\bbb} X_{j_2}^{\bbb} \ldots  X_{j_k}^{\bbb} \text{,}\\
   \left( X_{j_1} X_{j_2} \ldots X_{j_k}\right)^{\bbb, i} &= X_{j_1}^{\bbb, i} X_{j_2}^{\bbb, i} \ldots  X_{j_k}^{\bbb, i} \text{,}
\end{align*}
for all $j_1,\ldots,j_k \in \{1,\ldots, n\}$.

We are now in a position to define the polynomials that will play a central role in this paper.

\begin{definition}
A partially commutative polynomial 
\[
  f \in \pcp{F}{X_1,\ldots,X_n}{\Yo,\Y} 
\]
is said to be {\em admissible}
if there exists an $r\in\N_0$ (which we call the {\em order of $f$}) such that $f$ is either of the form
\begin{equation}\label{tone}
  f = \sum_{\sigma \in S_n} \sum_{\bbb \in \BBB{r}}\lamt{\sigma}{\bbb} \left( X_{\sigma(1)} X_{\sigma(2)} \ldots X_{\sigma(n)} \right)^{\bbb}
\end{equation}
for some  $\lamt{\sigma}{\bbb}\in F$ or of the form 
\begin{equation}\label{ttwo}
  f = \sum_{\sigma \in S_n} \sum_{\bbb \in \BBB{r}} \sum_{i=1}^n \lam{\sigma}{\bbb}{i} \left( X_{\sigma(1)} X_{\sigma(2)} \ldots X_{\sigma(n)} \right)^{\bbb, \sigma(i)}
\end{equation}
for some  $\lam{\sigma}{\bbb}{i}\in F$. If \eqref{tone} holds, then $f$ is said to be of {\em type one},
and if  \eqref{ttwo} holds, then $f$ is said to be of {\em type two}. 
\end{definition}

\begin{remark} \label{opomba21}
In the case where  $r=0$,  $\BBB{r}$  consists only of the sequence of zeros $\bbb$.
Therefore, in this case we have
\[
  \left( X_{\sigma(1)} X_{\sigma(2)} \ldots X_{\sigma(n)} \right)^{\bbb} = X_{\sigma(1)} X_{\sigma(2)} \ldots X_{\sigma(n)}.
\]
Thus, admissible partially commutative polynomials of type one and of order $0$ are of the form
\[
  f = \sum_{\sigma \in S_n} \lambda_{\sigma} X_{\sigma(1)} X_{\sigma(2)} \ldots X_{\sigma(n)},
\]
so these are exactly the multilinear noncommutative  polynomials. This is of crucial importance for us: any result on admissible
partially commutative polynomials yields a result on multilinear noncommutative  polynomials.
\end{remark}

To gain some feeling for the notions we have defined, we record a simple example.

\begin{example} \label{primer pzv}
Let us describe admissible
partially commutative polynomials in $\pcp{F}{X_1, X_2}{\Yo, \Y}$ of order $1$.
We have
\begin{align*}
  \BBB{1} = \{ (1, 0),~(0, 1) \} \text{.}
\end{align*}
Note that
\begin{align*}
  &\left( X_{1} X_{2} \right)^{(1,0)} = [\Y, X_{1}] X_{2}\text{,}  \quad\,\,\,\,\,    \left( X_{2} X_{1} \right)^{(1,0)} =X_{2} [\Y, X_{1}] \text{,} \\
  &\left( X_{1} X_{2} \right)^{(0,1)} = X_{1} [\Y, X_{2}]\text{,} \quad\,\,\,\,\,  \left( X_{2} X_{1} \right)^{(0,1)} = [\Y, X_{2}] X_{1}\text{,}
\end{align*}
and thus
\begin{align*}
  &\left( X_{1} X_{2} \right)^{(1,0), 1} = [\Yo, \Y, X_{1}] X_{2}\text{,}  \quad\,\,\,\,\,    \left( X_{2} X_{1} \right)^{(1,0), 1} =X_{2} [\Yo, \Y, X_{1}] \text{,} \\
  &\left( X_{1} X_{2} \right)^{(0,1), 1} = [\Yo, X_{1}] [\Y, X_{2}]\text{,} \quad\,\,  \left( X_{2} X_{1} \right)^{(0,1), 1} = [\Y, X_{2}] [\Yo, X_{1}]\text{,}\\
  &\left( X_{1} X_{2} \right)^{(1,0), 2} = [\Y, X_{1}] [\Yo, X_{2}]\text{,} \quad \,\, \left( X_{2} X_{1} \right)^{(1,0), 2} = [\Yo, X_{2}] [\Y, X_{1}] \text{,}\\
  &\left( X_{1} X_{2} \right)^{(0,1), 2} = X_{1} [\Yo, \Y, X_{2}]\text{,}         \quad \,\,\,\,\, \left( X_{2} X_{1} \right)^{(0,1), 2} = [\Yo, \Y, X_{2}] X_{1} \text{.}  
\end{align*}
Hence, admissible polynomials of type two are of the form
\begin{align*}
  f &=  \lam{{\rm id}}{(1,0)}{1} [\Yo, \Y, X_{1}] X_{2} + \lam{(12)}{ (1,0)}{2} X_{2} [\Yo, \Y, X_{1}] \\
  &+  \lam{{\rm id}}{(0,1)}{1} [\Yo, X_{1}] [\Y, X_{2}] + \lam{(12)}{ (0,1)}{2} [\Y, X_{2}] [\Yo, X_{1}] \\
  &+  \lam{{\rm id}}{(1,0)}{2} [\Y, X_{1}] [\Yo, X_{2}] + \lam{(12)}{ (1,0)}{1} [\Yo, X_{2}] [\Y, X_{1}] \\
  &+  \lam{{\rm id}}{(0,1)}{2} X_{1} [\Yo, \Y, X_{2}] + \lam{(12)}{ (0,1)}{1} [\Yo, \Y, X_{2}] X_{1}\text{,} 
\end{align*}
while  admissible polynomials of type one are of the form
\begin{align*}
  f &=  \lamt{{\rm id}}{(1,0)} [\Y, X_{1}] X_{2} + \lamt{(12)}{(1,0)} X_{2} [\Y, X_{1}] \\
  &+  \lamt{{\rm id}}{(0,1)} X_{1} [\Y, X_{2}] + \lamt{(12)}{(0,1)} [\Y, X_{2}] X_{1} \text{.}
\end{align*}
\end{example}

By definition, the vector space of admissible
partially commutative polynomials 
of type two (resp. type one) 
is linearly spanned by the polynomials $ \left( X_{\sigma(1)}  \ldots X_{\sigma(n)} \right)^{\bbb, \sigma(i)}$ (resp.  $\left( X_{\sigma(1)} \ldots X_{\sigma(n)} \right)^{\bbb}$).
Our goal now is to show that they are linearly independent and so they actually form a basis. To this end, we define
\[
  B_1 = \left\{ X_{j_1}\ldots X_{j_d} \,\, \middle\vert \,\, d \in \N,~j_1,\ldots,j_d \in \{1,\ldots, n\} \right\}
\]
and
\[
  B_2 = \left\{ \Yo^k \Y^l \,\, \middle\vert \,\, k, l \in \No,~k + l >0 \right\} \text{.}
\]
Of course, $B_1\cup \{1\}$ is the standard basis of $F\langle X_1,\ldots, X_n \rangle$ and 
$B_2\cup \{1\}$ is the standard basis of $F[\Yo, \Y]$. Letting $M$ be the set of all alternating monomials from $B_1$ and $B_2$,
we see from \cite[Lemma 1.4.5]{Rings}
that $M\cup \{1\}$ is  a basis of the vector space 
\[
    \pcp{F}{X_1,\ldots,X_n}{\Yo,\Y}
\]
of partially commutative polynomials. Using this, it is easy to see that the eight (resp. four) polynomials occurring in Example \ref{primer pzv}
are linearly independent. We now proceed to the general case.

\begin{proposition}
\label{lin neodvisnost}
For any  $n \in \N$ and $r \in \No$,
\[
  \left\{ \left( X_{\sigma(1)} X_{\sigma(2)} \ldots X_{\sigma(n)} \right)^{\bbb} \,\, \middle\vert \,\, \sigma \in S_n, ~\bbb \in \BBB{r} \right\}
\]
and
\[
  \left\{ \left( X_{\sigma(1)} X_{\sigma(2)} \ldots X_{\sigma(n)} \right)^{\bbb, \sigma(i)} \,\, \middle\vert \,\, \sigma \in S_n,  ~\bbb \in \BBB{r},~ i =1,\ldots,n \right\}
\]
are linearly independent sets.
\end{proposition}

\begin{proof}
We will prove the linear independence of the second set; the  proof for the first set is similar. Suppose 
\[
  \sum_{\sigma \in S_n} \sum_{\bbb \in \BBB{r}} \sum_{i=1}^n \lam{\sigma}{\bbb}{i} \left( X_{\sigma(1)} X_{\sigma(2)} \ldots X_{\sigma(n)} \right)^{\bbb, \sigma(i)} = 0 \text{,}
\]
where some of the  scalars $\lam{\sigma}{\bbb}{i}$ are nonzero.
Without loss of generality, we may assume that the set
\[
  \mathcal{L} = \left\{ (\bbb, i) \in \BBB{r} \times \{ 1, \ldots, n\} \,\, \middle\vert \,\, \lam{{\rm id}}{\bbb}{i} \neq 0\right\}
\]
is nonempty. Define the following  strict total ordering on $\mathcal{L}$:
\begin{align*}
  &(\bbb, i) \prec (\bbb', i') \\
  \iff \,& i < i' ~\text{or}~\left(i = i' ~\text{and}~ \bbb_j < \bbb'_j~\text{for the least $j$ satisfying}~ \bbb_j \neq \bbb'_j \right) \text{.}
\end{align*}
Let  $(\widetilde{\bbb}, \widetilde{i})$ be a  maximal element of $\mathcal{L}$ with respect to $\prec$.  We claim that by rewriting the above summation 
as a linear combination of elements of  $M$, one arrives at a contradiction that $\lam{{\rm id}}{\widetilde{\bbb}}{\widetilde{i}} = 0$. The formal proof is a bit tedious but elementary, so we omit the details.
\end{proof}

\section{A certain system of linear equations}\label{s3}

Let $\sigma \in S_n$  and let $r \in \No$. 
The purpose of this section is to examine the following  system of linear equations:
\begin{align} \label{grozni sistem enacb}
  \sum_{i=1}^n \sum_{\ccc \in \CCC{\sigma}{k}{i}} \skal{\sigma}{\ccc}{i} \mmt{\bbb - \ccc}{i} = 0
\end{align}
for every $k \in \N$ and every $\bbb \in \BBB{r+k}$, 
where 
\[
  \CCC{\sigma}{k}{i} = \left\{ \ccc = (\ccc_1, \ldots, \ccc_n) \in \N_0^n \,\, \middle\vert \,\, \ccc_{\sigma(1)} = \ldots = \ccc_{\sigma(i-1)} = 0,~ \ccc_{\sigma(i)} \geq 1,~ \sum_{j=1}^n \ccc_{j} = k\right\} \text{,}
\]
\[
  \skal{\sigma}{\ccc}{i} = {\sum_{j=1}^n \ccc_j \choose \ccc_{\sigma(i)}, \ldots, \ccc_{\sigma(n)}} = \frac{\left( \sum_{j=1}^n \ccc_{j} \right)!}{\ccc_{\sigma(i)}! \ldots \ccc_{\sigma(n)}!} \in\N \text{,}
\]
and $\mmt{\bbb'}{i}$, $\bbb' \in \BBB{r}$, $ i =1,\ldots, n$, are the unknowns, with
\begin{equation*}\label{co}
\mbox{$\mmt{\bbb'}{i} =0$ whenever $\bbb' \in \Z^n \setminus \BBB{r}$.}
\end{equation*}
The only property of the scalars $\skal{\sigma}{\ccc}{i}$ that we actually need   is that $$\skal{\sigma}{k \e{\sigma(i)}}{i} = 1,$$
where
$$ \e{j} = (0,\dots,0,1,0,\dots,0)$$
with $1$ appearing at the $j$th place.

The sole reason for discussing the  system of equations \eqref{grozni sistem enacb} is that it
 will occur in the proof of the main theorem.
We have introduced it in the exact form  as needed. However, it is clear that there is no loss of generality in assuming that $$\sigma = {\rm id} \text{.}$$ We will  write $\skalt{\ccc}{i}$ for 
$\skal{{\rm id}}{\ccc}{i}$ and 
$\CCCt{k}{i}$
for $\CCC{{\rm id}}{k}{i}$.

The goal of this section is to show that our system of equations has only trivial solution $\mmt{\bbb'}{i} = 0$ for all $\bbb' \in \BBB{r}$, $ i =1,\ldots, n$. We illustrate our method of solving  by a simple example.

\begin{example} \label{primer sistema}
Let  $n=2$ and $r = 1$. We have $\BBB{1} = \{(1, 0),~ (0, 1) \}$. Thus, our unknowns are $\mmt{(1, 0)}{1}$, $\mmt{(0, 1)}{1}$, $\mmt{(1, 0)}{2}$, and $\mmt{(0, 1)}{2}$.

First, let $k=1$. Note that $\BBB{2} = \{(2, 0),~ (1, 1),~(0, 2) \}$, $\CCCt{1}{1} = \{ (1, 0) \}$, $\CCCt{1}{2} = \{ (0, 1) \}$ and $\skalt{(1, 0)}{1} = \skalt{(0, 1)}{2} = 1$.
Thus, we have the equations
\begin{align*}
  \mmt{(1, 0)}{1} + \mmt{(2, -1)}{2}  &= 0 \text{,}\\
  \mmt{(0, 1)}{1} + \mmt{(1, 0)}{2}  &= 0 \text{,} \\
  \mmt{(-1, 2)}{1} + \mmt{(0, 1)}{2}  &= 0 \text{,}
\end{align*}
for $b = (2, 0),~ (1,1),~ (0, 2)$, respectively, with 
 $\mmt{(2, -1)}{2}=\mmt{(-1, 2)}{1}=0$.
By adding the equation $\mmt{(2, -1)}{1} + \mmt{(3, -2)}{2} = 0$, which trivially holds since $\mmt{(2, -1)}{1} = \mmt{(3, -2)}{2} = 0$, we extend the above system to
\begin{align} \label{sistem k1}
\begin{split}
  \mmt{(2, -1)}{1} + \mmt{(3, -2)}{2} &= 0 \text{,} \\
  \mmt{(1, 0)}{1} + \mmt{(2, -1)}{2}  &= 0 \text{,}\\
  \mmt{(0, 1)}{1} + \mmt{(1, 0)}{2}  &= 0 \text{,} \\
  \mmt{(-1, 2)}{1} + \mmt{(0, 1)}{2}  &= 0 \text{.}
\end{split}
\end{align}

Now, let  $k=2$. Note that $\BBB{3} = \{(3, 0),~ (2, 1),~(1, 2),~(0,3) \}$, $\CCCt{2}{1} = \{ (2, 0),~(1, 1) \}$ and $\CCCt{2}{2} = \{ (0, 2) \}$. Also, $\skalt{(2, 0)}{1} = 1$, $\skalt{(1, 1)}{1} = 2$ and $\skalt{(0, 2)}{2} = 1$. Thus, we have the equations
\begin{align} \label{sistem k2}
\begin{split}
  \mmt{(1, 0)}{1} + 2 \mmt{(2, -1)}{1} + \mmt{(3, -2)}{2}  &= 0 \text{,}\\
  \mmt{(0, 1)}{1} + 2 \mmt{(1, 0)}{1} + \mmt{(2, -1)}{2}  &= 0 \text{,} \\
  \mmt{(-1, 2)}{1} + 2 \mmt{(0, 1)}{1} + \mmt{(1, 0)}{2}  &= 0 \text{,} \\
  \mmt{(-2, 3)}{1} + 2 \mmt{(-1, 2)}{1} + \mmt{(0, 1)}{2}  &= 0 \text{,}
\end{split}
\end{align}
for $b = (3, 0),~ (2,1),~ (1, 2),~(0, 3)$, respectively, with 
 $\mmt{(-2, 3)}{1} =0$.

By substracting each line of the system \eqref{sistem k1} from the corresponding line in \eqref{sistem k2}, we obtain
\begin{align} \label{zadnji sistem}
\begin{split}
   \mmt{(1, 0)}{1} + \mmt{(2, -1)}{1} &= 0 \text{,} \\
   \mmt{(0, 1)}{1} + \mmt{(1, 0)}{1} &= 0 \text{,} \\
   \mmt{(-1, 2)}{1} + \mmt{(0, 1)}{1} &= 0 \text{,} \\
   \mmt{(-2, 3)}{1} + \mmt{(-1, 2)}{1} &= 0 \text{.}
\end{split}
\end{align}
Since $\mmt{(2, -1)}{1} = 0$, the first equation yields $\mmt{(1, 0)}{1}=0$. Accordingly, the second equation gives $\mmt{(0, 1)}{1}=0$. From 
\eqref{sistem k1} we  infer that $\mmt{(1, 0)}{2} = \mmt{(0, 1)}{2} =0$. We have thus shown that  (\ref{grozni sistem enacb}) 
has only trivial solution for $n =2$ and $r= 1$.
\end{example}

As we can see from Example \ref{primer sistema}, the number of equations is different for different $k$. We remedy this by adding trivial equations. That is, we extend the system \eqref{grozni sistem enacb} to
\begin{align} \label{razsirjen sistem}
\sum_{i=1}^n \sum_{\ccc \in \CCCt{k}{i}} \skalt{\ccc}{i} \mmt{\bbb - \ccc}{i} = 0
\end{align}
for every $k \in \N$ and every $\bbb \in \Z^n$.

	Let us check  that any solution of the
	restricted system of equations \eqref{grozni sistem enacb} is also a solution of the 
	extended system \eqref{razsirjen sistem}.
	For this purpose, we fix $k \in \No$ and $\bbb \in \Z^n \setminus \BBB{r+k}$, i.e., $\bbb \in \Z^n$ and either $\sum_{j=1}^n \bbb_j \neq r + k$ or $\bbb_j < 0$ for some $j$.
	If $\sum_{j=1}^n \bbb_j \neq r + k$, then $\sum_{j=1}^n (\bbb - \ccc)_j \neq r$ for all $\ccc \in \CCCt{k}{i}$, since $\sum_{j=1}^n \ccc_j = k$. Thus, $\bbb - \ccc \notin \BBB{r}$ for all $\ccc \in \CCCt{k}{i}$; hence, all $\mmt{\bbb - \ccc}{i}$ are zero and  \eqref{razsirjen sistem} holds. Similarly, we arrive at the same conclusion in the case where $\bbb_j < 0$ for some $j$.

Let $V$ be the $F$-vector space with basis $\{ \nabs{\bbb}{i} \mid \bbb \in \Z^n,~ i =1,\ldots, n \}$ (here,
$\nabs{\bbb}{i}$ are just abstract symbols).
 For every $j = 1,\ldots, n$, let $\lm{j}: V \rightarrow V$ be the linear map given by
\[
  \lm{j}:  \nabs{\bbb}{i} \mapsto \nabs{\bbb - \e{j}}{i} \text{.}
\]
Note that the  maps $\lm{j}$ commute among themselves.

We make a connection between  $\nabs{\bbb}{i}$ and $\mmt{\bbb}{i}$ through the linear functional $\phi: V \rightarrow F$ given by
$$\phi \left( \nabs{\bbb}{i} \right) = \mmt{\bbb}{i} \text{.}$$
We can  rewrite the system \eqref{razsirjen sistem} as 
\begin{align*}
  \sum_{i=1}^n \sum_{\ccc \in \CCCt{k}{i}} \skalt{\ccc}{i} \nabs{\bbb - \ccc}{i} \in \ker \phi
\end{align*}
for every   $k \in \N$ and every $\bbb \in \Z^n$. Since $\ccc \in \CCCt{k}{i}$ can be written as $\ccc = \ccc_i \e{i} + \cdots + \ccc_n \e{n}$, we have
$\nabs{\bbb - \ccc}{i} = \lm{i}^{\ccc_i} \ldots \lm{n}^{\ccc_n} \nabs{\bbb}{i}$.
Thus, we can further rewrite our system as
\begin{align*}
  \sum_{i=1}^n \left( \sum_{\ccc \in \CCCt{k}{i}} \skalt{\ccc}{i} \lm{i}^{\ccc_i} \ldots \lm{n}^{\ccc_n} \right) \nabs{\bbb}{i} \in \ker \phi
\end{align*}
for every   $k \in \N$ and every $\bbb \in \Z^n$.  The expression in parentheses can be viewed  as a (commutative) polynomial evaluated in  $\lm{i}, \ldots, \lm{n}$. Since $\skalt{k \e{i}}{i} = 1$ and $\sum_{j=1}^n \ccc_j = k$ for every $\ccc \in \CCCt{k}{i}$, the  degree of this polynomial in the $i$th variable   is exactly $k$.
This motivates the following definition.

\begin{definition}
Let $U$ be a subspace of $V$ and let $i_1, \ldots, i_l \in \{ 1, \ldots, n \}$ be distinct. We say that  $\nabs{\bbb}{i_1}, \ldots, \nabs{\bbb}{i_l}$ \emph{satisfy a recurrence relation of order $k \in \N$} if
\[
  \sum_{i \in \{i_1, \ldots, i_l \}} f_{i}(\lm{i}, \ldots, \lm{n}) \nabs{\bbb}{i} \in U
\]
holds for every $\bbb \in \Z^n$ and  some  polynomials $f_{i} \in F[W_i, \ldots, W_n]$ such that the degree of $f_i$ in  $W_i$ is $k$.
\end{definition}

We summarize our observations up to this point. Suppose $\mmt{\bbb}{i} \in F$ satisfy  \eqref{razsirjen sistem}, where $\mmt{\bbb}{i} = 0$ for all $\bbb \not\in \BBB{r}$, $i = 1,\ldots, n$. By taking $U = \ker\phi$ with $\phi$ defined above, the elements $\nabs{\bbb}{1}, \ldots, \nabs{\bbb}{n}$ satisfy a recurrence relation of order $k$ for each $k \in \N$. Since only finitely many $\mmt{\bbb}{i}$ can be nonzero, $\nabs{\bbb}{i} \in U$ holds for all but finitely many elements. We claim that this already implies that $U = V$ or, equivalently, $\mmt{\bbb}{i} = 0$ for all $\bbb \in \BBB{r}$, $i = 1,\ldots, n$.
To this end, we need a few lemmas.

\begin{lemma} \label{lema o mnozenju}
Let $U$ be a subspace of $V$. 
If $f_{i} \in F[W_i, \ldots, W_n]$ satisfy 
\begin{align} \label{random enacba}
  \sum_{i \in \{i_1, \ldots, i_l \}} f_{i}(\lm{i}, \ldots, \lm{n}) \nabs{\bbb}{i} \in U
\end{align}
for every $\bbb \in \Z^n$, then
\[
  \sum_{i \in \{i_1, \ldots, i_l \}} g(\lm{1}, \ldots, \lm{n}) f_{i}(\lm{i}, \ldots, \lm{n}) \nabs{\bbb}{i} \in U
\]
 for every $\bbb \in \Z^n$ and every $g \in F[W_1, \ldots, W_n]$.
\end{lemma}

\begin{proof}
Since $U$ is a vector space, it suffices to consider the case where $g$ is a monomial, so 
\[
g(\lm{1}, \dots, \lm{n}) = \lm{1}^{j_1} \ldots \lm{n}^{j_n}.
\]
Take $\bbb \in \Z^n$.
Using the commutativity of the $\lm{j}$s, we have
\begin{align*}
  \sum_{i \in \{i_1, \ldots, i_l \}} g(\lm{1}, \ldots, \lm{n}) f_{i}(\lm{i}, \ldots, \lm{n}) \nabs{\bbb}{i} &= \sum_{i \in \{i_1, \ldots, i_l \}} f_{i}(\lm{i}, \ldots, \lm{n}) \lm{1}^{j_1} \ldots \lm{n}^{j_n} \nabs{\bbb}{i} \\
  &= \sum_{i \in \{i_1, \ldots, i_l \}} f_{i}(\lm{i}, \ldots, \lm{n}) \nabs{\bbb - j_1 \e{1} - \cdots - j_n \e{n}}{i} \text{.}
\end{align*}
Since \eqref{random enacba} holds for
$\bbb' = \bbb - j_1 \e{1} - \cdots - j_n \e{n}$,
the desired conclusion follows.
\end{proof}

The next lemma 
 generalizes the process of subtracting equations in  Example \ref{primer sistema} (and 
is essentially a version of  Gaussian elimination).

\begin{lemma} \label{lema32} Let $U$ be a subspace of  $V$ and let
$1\le i_1 < \dots < i_l  \le n$. 
Suppose  $\nabs{\bbb}{i_1}, \ldots, \nabs{\bbb}{i_l}$ satisfy $l$ recurrence relations of order $k_1 > \dots > k_l$. Then $\nabs{\bbb}{i_1}$ satisfy a recurrence relation (of order $k_1$).
\end{lemma}

\begin{proof}
We proceed by induction on $l$. There is nothing to prove if  $l=1$, so assume that $l > 1$ and that the lemma is true for all numbers smaller than $l$.
 By assumption,
\[
  \sum_{i \in \{i_1, \ldots, i_l \}} f_{i j}(\lm{i}, \ldots, \lm{n}) \nabs{\bbb}{i} \in U
\]
for every $\bbb \in \Z^n$ and every $j = 1, \ldots, l$, where $f_{i j} \in F[W_i, \ldots, W_n]$ are polynomials such that the   degree of $f_{ij}$ in $W_i$ is $k_j$.

Fix $j \in \{1, \ldots, l - 1 \}$ and $\bbb \in \Z^n$.
By Lemma \ref{lema o mnozenju}, we have
\begin{align*}
   &\sum_{i \in \{i_1, \ldots, i_l \}} f_{i_l j}(\lm{i_l}, \ldots, \lm{n}) f_{i l}(\lm{i}, \ldots, \lm{n}) \nabs{\bbb}{i} \in U \text{,}\\
   &\sum_{i \in \{i_1, \ldots, i_l \}} f_{i_l l}(\lm{i_l}, \ldots, \lm{n}) f_{i j}(\lm{i}, \ldots, \lm{n}) \nabs{\bbb}{i} \in U \text{.}
\end{align*}
Subtracting yields
\[
  \sum_{i \in \{i_1, \ldots, i_l \}} g_{i j}(\lm{i}, \ldots, \lm{n}) \nabs{\bbb}{i} \in U \text{,}
\]
where
\begin{align*}
  &g_{i j}(W_i, \ldots, W_n)\\
  =\,& f_{i_l l}(W_{i_l}, \ldots, W_n) f_{i j}(W_i, \ldots, W_n) - f_{i_l j}(W_{i_l}, \ldots, W_n) f_{i l}(W_i, \ldots, W_n) \text{.}
\end{align*}
Note that $g_{i_l j} = 0$. For every $i \in \{ i_1, \ldots, i_{l-1} \}$, we have $i < i_l$. Hence,  the variable $W_i$ does not occur in the 
polynomials $f_{i_l l}(W_{i_l}, \ldots, W_n)$ and $f_{i_l j}(W_{i_l}, \ldots, W_n)$.
Since the  degree of the  polynomials $f_{i j}(W_i, \ldots, W_n)$ and $f_{i l}(W_i, \ldots, W_n)$ in $W_i$  is $k_j$ and $k_l$, respectively, with $k_j > k_l$, and the polynomial $f_{i_l l}(W_{i_l}, \ldots, W_n)$ is nonzero, the  degree of the  polynomial $g_{i j}(W_i, \ldots, W_n)$ in $W_i$ is $k_j$.

As $j \in \{1, \ldots, l-1\}$ and $\bbb \in \Z^n$ were arbitrary, we have
\[
  \sum_{i \in \{i_1, \ldots, i_{l-1} \}} g_{i j}(\lm{i}, \ldots, \lm{n}) \nabs{\bbb}{i} \in U 
\]
for every $\bbb \in \Z^n$ and every $j \in \{1, \ldots, l-1\}$. Since the  degree of $g_{i j}(W_i, \ldots, W_n)$ in  $W_i$   is $k_j$, $\nabs{\bbb}{i_1}, \ldots, \nabs{\bbb}{i_{l-1}}$ satisfy $l-1$ recurrence relations of order $k_1 > \dots > k_{l-1}$. We may now use the induction hypothesis 
and the lemma follows.
\end{proof}

The next lemma 
 generalizes the process  of solving  the system of equations \eqref{zadnji sistem} in Example \ref{primer sistema}.

\begin{lemma} \label{lema33}
Let $i \in \{1,\ldots, n\}$  and  let $U$  be a subspace of $V$ such that $\nabs{\bbb}{i} \in U$  for all but  finitely many  $\bbb \in \Z^n$. If $\nabs{\bbb}{i}$ satisfy a recurrence relation, then $\nabs{\bbb}{i} \in U$  for all $\bbb \in \Z^n$.
\end{lemma}

\begin{proof}
Let 
\[
  f(W_i, \ldots, W_n) = \sum_{j_i, \ldots, j_n \geq 0} \lambda_{j_i, \ldots, j_n} W_i^{j_i} \ldots W_n^{j_n}
\]
be such that
\begin{align} \label{prva enacba}
  f(\lm{i}, \ldots, \lm{n}) \nabs{\bbb}{i} \in U
\end{align}
for every $\bbb \in \Z^n$.
The set
\[
  J = \left\{ (j_i, \ldots, j_n) \in \No^{n-i+1} \,\, \middle\vert \,\, \lambda_{j_i, \ldots, j_n} \neq 0\right\} 
\]
is nonempty and finite. We can write  $f$ as
\[
  f(W_i, \ldots, W_n) = \sum_{(j_i, \ldots, j_n) \in J} \lambda_{j_i, \ldots, j_n} W_i^{j_i} \ldots W_n^{j_n} \text{.}
\]
Let $(\widetilde{j}_i, \ldots, \widetilde{j}_n)$ be a maximal element of $J$ with respect to the lexicographical ordering.

The set
$S = \{ \bbb \in \Z^n \mid \nabs{\bbb}{i} \notin U\}$
is finite.
Suppose $S$ is nonempty. Define the following strict partial ordering on $S$:
\begin{align*}
  \bbb \prec \bbb' \iff \bbb_k < \bbb'_k ~\text{for the least $k \geq i$ satisfying $\bbb_k \neq \bbb'_k$} \text{.}
\end{align*}
We remark that  this ordering is lexicographic in the last $n-i+1$ terms. Let $\widetilde{\bbb}$ be a maximal element of $S$ with respect to $\prec$.

Set
\[
  \bbb' = \widetilde{\bbb} + \widetilde{j_i} \e{i} + \cdots + \widetilde{j}_n \e{n} \text{.}
\]
We claim that there exists a $u \in U$ such that
\begin{align}
\label{druga enacba}
  f(\lm{i}, \ldots, \lm{n}) \nabs{\bbb'}{i} = \lambda_{\widetilde{j}_i, \ldots, \widetilde{j}_n} \nabs{\widetilde{\bbb}}{i} + u.
\end{align}
Indeed, by definition of $\lm{j}$, 
\begin{align*}
  f(\lm{i}, \ldots, \lm{n}) \nabs{\bbb'}{i} &= \sum_{(j_i, \ldots, j_n) \in J} \lambda_{j_i, \ldots, j_n} \lm{i}^{j_i} \ldots \lm{n}^{j_n}\nabs{\bbb'}{i} \\
  &=  \sum_{(j_i, \ldots, j_n) \in J} \lambda_{j_i, \ldots, j_n}\nabs{\bbb' - j_i \e{i} - \cdots - j_n \e{n}}{i} \text{.}
\end{align*}
Take $(j_i, \ldots, j_n) \in J$ distinct from $(\widetilde{j}_i, \ldots, \widetilde{j}_n)$. Since $(\widetilde{j}_i, \ldots, \widetilde{j}_n)$ is a maximal element of $J$ and the lexicographical ordering is linear, we have $j_k < \widetilde{j}_k$ for the least $k$ satisfying $j_k \neq \widetilde{j}_k$.
Hence
\begin{align*}
  \bbb &= \bbb' - j_i \e{i} - \cdots - j_n \e{n} \\
  &= \widetilde{\bbb} + (\widetilde{j_i} - j_i) \e{i} + \cdots + (\widetilde{j}_n - j_n) \e{n} \text{,}
\end{align*}
satisfies
\[
  \widetilde{\bbb} \prec \bbb
\]
if $\bbb \in S$.
As $\widetilde{\bbb}$ is a maximal element of $S$, it follows that $\bbb \notin S$, i.e., $\nabs{\bbb}{i} \in U$. Thus, all summands of
\[
  \sum_{(j_i, \ldots, j_n) \in J} \lambda_{j_i, \ldots, j_n}\nabs{\bbb' - j_i \e{i} - \cdots - j_n \e{n}}{i}
\]
except the one corresponding to $(\widetilde{j}_i, \ldots, \widetilde{j}_n)$
are elements of $U$. This proves our claim \eqref{druga enacba}.

From \eqref{prva enacba} and \eqref{druga enacba} we deduce
\[
  \lambda_{\widetilde{j}_i, \ldots, \widetilde{j}_n} \nabs{\widetilde{\bbb}}{i} \in U \text{.}
\]
Since $\widetilde{\bbb} \in S$, we have $\nabs{\widetilde{\bbb}}{i} \notin U$. Thus, $\lambda_{\widetilde{j}_i, \ldots, \widetilde{j}_n} = 0$, which is a contradiction to $(\widetilde{j}_i, \ldots, \widetilde{j}_n) \in J$.
\end{proof}

Recall that we only need to prove that $U=V$, having in mind that  $\nabs{\bbb}{1}, \ldots, \nabs{\bbb}{n}$ satisfy a recurrence relation of order $k$ for each $k \in \N$ and that $\nabs{\bbb}{i} \in U$ holds for all but finitely many elements. This easily follows from repeated use of Lemmas \ref{lema32} and \ref{lema33}.
Thus, we have proven the following proposition, which concludes this section.

\begin{proposition} \label{grozna trditev}
Let $\sigma \in S_n$ and $r \in \No$. If $\mmt{\bbb'}{i} \in F$ are such that $\mmt{\bbb'}{i} = 0$ whenever $\bbb' \not\in \BBB{r}$ and
\begin{align*}
  \sum_{i=1}^n \sum_{\ccc \in \CCC{\sigma}{k}{i}} \skal{\sigma}{\ccc}{i} \mmt{\bbb - \ccc}{i} = 0
\end{align*}
for every $k \in \N$ and every $\bbb \in \BBB{r+k}$, then $\mmt{\bbb'}{i} = 0$  for all $\bbb' \in \BBB{r}$, $i = 1,\ldots, n$.
\end{proposition}

\section{Main theorem}\label{s4}

First, we fix some notation. Throughout this section, we assume that $A$ is a unital algebra containing an element $\y$ with the property that 
$$1\in \bigcap_{k\in\N}[\y,A]_k \text{.}$$ That is to say, 
for each 
$k \in \N$ there exists a $z_k \in A$ such that
\begin{equation}\label{e1}
[\underbrace{\y, \ldots, \y}_k, z_k] = 1 \text{.}
\end{equation}
Let us emphasize that this is
the only requirement we impose on $A$. Thus,   $A$ does not need to have a surjective inner derivation, but only an element $\y$ satisfying \eqref{e1} for every 
$k\in \mathbb N$.
Observe  that under the assumption that $F$ has characteristic $0$
it is enough  to require only the existence of $z_1$, since then $z_k = \frac{1}{k!}z_1^k$ automatically satisfies \eqref{e1} for any $k$.

We will write $$\Ao=\bigcap_{k\in\N}[\y,A]_k \text{.}$$  Thus, our assumption is that $1\in \Ao$. 
Our goal is to prove Theorem \ref{mt} which states that the image 
of any nonzero admissible partially commutative polynomial $f$ 
contains $\Ao$. If 
the derivation $x\mapsto [\y,x]$ is surjective, then  $\Ao=A$ and so this  simply means that $f$ is surjective on $A$.

The  proof is by induction on the number of noncommuting variables $X_1, \ldots, X_n$. 
We start with  the base case $n=1$.

\begin{lemma} \label{baza}
	Let
	$f \in \pcp{F}{X_1}{\Yo,\Y}$
	be a nonzero admissible partially commutative polynomial. Then
	$\Ao \subseteq f(A; \y)$.
\end{lemma}

\begin{proof}
	We will only prove the lemma for admissible polynomials of type two; the type one case is similar. Since $n=1$,
	$\BBB{r} = \{ (r)\}$.
	Thus,  $f$ is equal to
	\[
	f = \lambda X_1^{(r),1} = \lambda [\Yo, [\Y, X_1]_r]
	\]
	for some nonzero $\lambda \in F$. Let $a$ be an arbitrary element in $\Ao$. By the definition of $\Ao$, there exists an $x \in A$ such that
	$a = [\y, x]_{r+1}$.
	As the image of $f$ contains
	\begin{align*}
	\Ev{\lambda^{-1} x}{\y}(f) = \lambda [\y, [\y, \lambda^{-1} x]_r] = a \text{,}
	\end{align*}
	$\Ao \subseteq f(A; \y)$ follows.
\end{proof}

To make the  induction step,
we first record  a  lemma which will help us  reduce the number of noncommuting variables in a suitable way.

\begin{lemma} \label{x v y}
	Let
	$f \in \pcp{F}{X_1,\ldots,X_n}{\Yo,\Y}$
	be a polynomial of the form
	\[
	f = \sum_{\sigma \in S_{n-1}} \sum_{j = 1}^n \sum_{\bbb \in \BBB{r}}\lama{\sigma}{\bbb}{j} \left( X_{\sigma(1)} \ldots X_{\sigma(j-1)} X_{n} X_{\sigma(j)} \ldots X_{\sigma(n-1)} \right)^{\bbb} \text{,}
	\]
	where  $r \in \N_0$. Let $k \in \{0, 1, \ldots, r \}$ be  such that for every $\bbb \in \BBB{r}$, $\bbb_n < k$ implies $\lama{\sigma}{\bbb}{j} = 0$ for every  $\sigma \in S_{n-1}$ and every $j \in \{1, \ldots, n\}$. Then the partially commutative polynomial
	$g \in \pcp{F}{X_1,\ldots,X_n}{\Yo,\Y}$
	defined by
	\[
	g = \sum_{\sigma \in S_{n-1}} \sum_{j=1}^n \sum_{\bbbi \in \BBBi{r}}\lama{\sigma}{\bbbi}{j} \left( X_{\sigma(1)} \ldots X_{\sigma(j-1)} \right)^{\bbbi} \Yo \left( X_{\sigma(j)}  \ldots X_{\sigma(n-1)} \right)^{\bbbi} \text{,}
	\]
	where $\BBBi{r} = \{ \bbbi \in \BBB{r} \mid \bbbi_n = k\}$, satisfies
	$g(A; \y) \subseteq f(A; \y)$.
\end{lemma}

\begin{proof}
	Take $x_1,\ldots, x_{n-1} \in A$ and $\yo \in F[\y]$. Let $z_k \in A$ be the element satisfying \eqref{e1}.  We will prove that
	\[
	\Ev{x_1,\ldots, x_{n-1}}{\yo}(g) = \Ev{x_1,\ldots, x_{n-1},z_k \yo}{}(f) \text{,}
	\]
	from which $g(A; \y) \subseteq f(A; \y)$ follows.

	Consider the expression
	$\Ev{x_1,\ldots, x_{n-1},z_k \yo}{}(\lama{\sigma}{\bbb}{j} X_n^{\bbb})$
	with $\bbb \in \BBB{r}$, $\sigma \in S_{n-1}$, and  $j \in \{1, \ldots, n\}$.
	If $\bbb_n < k$, then $\lama{\sigma}{\bbb}{j}=0$ and so this expression is zero, and 
	if $\bbb_n \geq k$, then it is equal to
	\begin{align*}
	\Ev{x_1,\ldots, x_{n-1},z_k \yo}{}(\lama{\sigma}{\bbb}{j} X_n^{\bbb}) &= \lama{\sigma}{\bbb}{j} \Ev{x_1,\ldots, x_{n-1},z_k \yo}{}([\Y, X_n]_{\bbb_n}) \\
	&= \lama{\sigma}{\bbb}{j} [\y, z_k \yo]_{\bbb_n} \text{.}
	\end{align*}
	Since $\yo$ commutes with $\y$, we have
	\begin{align*}
	\Ev{x_1,\ldots, x_{n-1},z_k \yo}{}(\lama{\sigma}{\bbb}{j} X_n^{\bbb}) &=\lama{\sigma}{\bbb}{j} [\y, z_k]_{\bbb_n} \yo \\
	&= \lama{\sigma}{\bbb}{j} \big[\underbrace{\y, \ldots, \y}_{\bbb_n-k}, [\y, z_k]_k\big] \yo \\
	&= \lama{\sigma}{\bbb}{j} [\y, 1]_{\bbb_n-k} \yo \text{,}
	\end{align*}
	which is  zero if $\bbb_n > k$ and equal to  $\lama{\sigma}{\bbb}{j} \yo$ if $\bbb_n = k$.
	Thus, we have
	\[
	\Ev{x_1,\ldots, x_{n-1},z_k \yo}{}(\lama{\sigma}{\bbb}{j} X_n^{\bbb}) = \begin{cases} \lama{\sigma}{\bbb}{j} \yo,~\bbb_n = k \\ 0,~\bbb_n \neq k \end{cases} \text{.}
	\]

	Hence, 
	\begin{align*}
	&\Ev{x_1,\ldots, x_{n-1},z_k \yo}{}(f)\\
	=\,&\sum_{\sigma \in S_{n-1}} \sum_{j=1}^n \sum_{\bbb \in \BBB{r}} \Ev{x_1,\ldots, x_{n-1},z_k \yo}{} \left( \left( X_{\sigma(1)} \ldots X_{\sigma(j-1)} \right)^{\bbb} \right) \Ev{x_1,\ldots, x_{n-1},z_k \yo}{} \left( \lama{\sigma}{\bbb}{j} X_n^{\bbb} \right)  \\ 
	&\cdot \Ev{x_1,\ldots, x_{n-1},z_k \yo}{} \left( \left( X_{\sigma(j)} \ldots X_{\sigma(n-1)} \right)^{\bbb} \right) \\
	=\,& \sum_{\sigma \in S_{n-1}} \sum_{j=1}^n \sum_{\bbbi \in \BBBi{r}} \Ev{x_1,\ldots, x_{n-1},z_k \yo}{} \left( \left( X_{\sigma(1)} \ldots X_{\sigma(j-1)} \right)^{\bbbi} \right) \lama{\sigma}{\bbbi}{j} \yo  \\ 
	& \cdot \Ev{x_1,\ldots, x_{n-1},z_k \yo}{} \left( \left( X_{\sigma(j)} \ldots X_{\sigma(n-1)}\right)^{\bbbi} \right) \text{.}
	\end{align*}
	Since $\yo = \Ev{x_1,\ldots, x_{n-1}}{\yo}(\Yo)$ and the variable $X_n$ does not appear in the polynomials $X_{\sigma(1)} \ldots X_{\sigma(j-1)}$ and $X_{\sigma(j)} \ldots X_{\sigma(n-1)}$, we see that $\Ev{x_1,\ldots, x_{n-1},z_k \yo}{}(f)$ is equal to
	\begin{align*}
	\sum_{\sigma \in S_{n-1}} \sum_{j=1}^n \sum_{\bbbi \in \BBBi{r}} \lama{\sigma}{\bbbi}{j} \Ev{x_1,\ldots, x_{n-1}}{\yo} \left( \left( X_{\sigma(1)} \ldots X_{\sigma(j-1)} \right)^{\bbbi} \Yo \left( X_{\sigma(j)} \ldots X_{\sigma(n-1)} \right)^{\bbbi} \right) \text{,}
	\end{align*}
	which is exactly $\Ev{x_1,\ldots, x_{n-1}}{\yo}(g)$.
\end{proof}

We are now in a position to  make the induction step for admissible polynomials of type one.

\begin{lemma} \label{korak}
Suppose $n\ge 2$ is such that $\Ao \subseteq g(A;\y)$ holds for every nonzero admissible partially commutative polynomial
	$g \in \pcp{F}{X_1,\ldots,X_{n-1}}{\Yo,\Y}$
	(of either type one or type two).
	Then
	$\Ao \subseteq f(A; \y)$ 
	for every nonzero admissible partially commutative polynomial
	$f \in \pcp{F}{X_1,\ldots,X_{n}}{\Yo,\Y}$
	of type one.
\end{lemma}

\begin{proof}
Suppose the lemma is not true. Then there exists a polynomial $f$ 
of the form 
	\[
	f = \sum_{\sigma \in S_n} \sum_{\bbb \in \BBB{r}}\lamt{\sigma}{\bbb} \left( X_{\sigma(1)} X_{\sigma(2)} \ldots X_{\sigma(n)} \right)^{\bbb} 
	\]
	where $\lamt{\sigma}{\bbb} \in F$ are not all zero, 
such that
	 $f(A; \y)$ does not contain $\Ao$. 
We rewrite $f$ as
	\[
	f = \sum_{\sigma \in S_{n-1}} \sum_{j=1}^n \sum_{\bbb \in \BBB{r}}\lama{\sigma}{\bbb}{j} \left( X_{\sigma(1)} \ldots X_{\sigma(j-1)} X_{n} X_{\sigma(j)} \ldots X_{\sigma(n-1)} \right)^{\bbb} \text{,}
	\]
	for some $\lama{\sigma}{\bbb}{j} \in F$, not all zero. 
	
	The set
	\[
	K = \{ k \in \No \mid \bbb_n = k ~\text{and}~ \lama{\sigma}{ \bbb}{j} \neq 0 \text{ for some $\sigma \in S_{n-1}$, $\bbb \in \BBB{r}$, $j = 1, \ldots, n$}\}
	\]
	is nonempty. Let $k$ be the least number in $K$. Then, for every $\bbb \in \BBB{r}$, $\bbb_n < k$ implies $\lama{\sigma}{\bbb}{ j} = 0$ for every  $\sigma \in S_{n-1}$ and every $j \in \{1, \ldots, n\}$.
	By Lemma \ref{x v y}, $g(A; \y) \subseteq f(A; \y)$ holds for the polynomial
	$g \in \pcp{F}{X_1,\ldots,X_{n-1}}{\Yo,\Y}$
	defined by
	\[
	g = \sum_{\sigma \in S_{n-1}} \sum_{j=1}^n \sum_{\bbbi \in \BBBi{r}}\lama{\sigma}{\bbbi}{j} \left( X_{\sigma(1)} \ldots X_{\sigma(j-1)} \right)^{\bbbi} \Yo \left( X_{\sigma(j)} \ldots X_{\sigma(n-1)} \right)^{\bbbi} \text{,}
	\]
	where $\BBBi{r} = \{ \bbbi \in \BBB{r} \mid \bbbi_n = k\}$.
	Since the variable $X_n$ does not occur in this expression, we can replace $\bbbi$ by the sequence $\bbb$ consisting of the first $n-1$ terms of $\bbbi$. Such sequences  $\bbb$ are exactly the elements of $\BBB{r- k}$. Thus, we have
	\begin{align*}
	g = \sum_{\sigma \in S_{n-1}} \sum_{j=1}^n \sum_{\bbb \in \BBB{r-k}}\lama{\sigma}{\bbbi}{j} \left( X_{\sigma(1)} \ldots X_{\sigma(j-1)} \right)^{\bbb} \Yo \left( X_{\sigma(j)} \ldots X_{\sigma(n-1)} \right)^{\bbb} \text{,}
	\end{align*}
	where $\BBB{r-k}$ consists of sequences of length $n-1$ and $\bbbi$ is the sequence of length $n$ obtained by adding $k$ to the end of $\bbb$.

	The polynomial $g$ may not be admissible. By using the homomorphism $\ho$ (defined in Section \ref{s2}), we obtain
	\[
	\ho(g) = \sum_{\sigma \in S_{n-1}} \sum_{\bbb \in \BBB{r-k}} \left( \sum_{j=1}^n \lama{\sigma}{\bbbi}{j} \right) \left( X_{\sigma(1)} \ldots X_{\sigma(j-1)} X_{\sigma(j)}  \ldots X_{\sigma(n-1)} \right)^{\bbb} \text{.}
	\]
	As we  see, $\ho(g)$ is an admissible partially commutative polynomial (of type one) in $n-1$ noncommuting variables.
	Since the image of $\ho(g)$ is a subset of $f(A; \y)$ (see \eqref{gg}), it cannot contain $\Ao$. Hence, by our assumption, $\ho(g)$ is zero.
	By Proposition \ref{lin neodvisnost},  $\sum_{j=1}^n \lama{\sigma}{\bbbi}{j} = 0$ for every  $\sigma \in S_{n-1}$ and every $\bbb \in \BBB{r-k}$.

	Considering $\lama{\sigma}{\bbbi}{n} = - \sum_{j=1}^{n-1} \lama{\sigma}{\bbbi}{j}$, we can write  $g$ as
	\begin{align*}
	g = \sum_{\sigma \in S_{n-1}} \sum_{\bbb \in \BBB{r-k}} \sum_{j=1}^{n-1} \lama{\sigma}{\bbbi}{j} \Big( &\left( X_{\sigma(1)} \ldots X_{\sigma(j-1)} \right)^{\bbb} \Yo \left( X_{\sigma(j)}  \ldots X_{\sigma(n-1)} \right)^{\bbb} \\
	-&  \left( X_{\sigma(1)} \ldots X_{\sigma(j-1)} X_{\sigma(j)}  \ldots X_{\sigma(n-1)} \right)^{\bbb} \Yo \Big) \text{,}
	\end{align*}
	which further equals
	\begin{align*}
	g = \sum_{\sigma \in S_{n-1}} \sum_{\bbb \in \BBB{r-k}} \sum_{j=1}^{n-1} \lama{\sigma}{\bbbi}{j} \left( X_{\sigma(1)} \ldots X_{\sigma(j-1)} \right)^{\bbb} \left[ \Yo, \left( X_{\sigma(j)}  \ldots X_{\sigma(n-1)} \right)^{\bbb} \right] \text{.}
	\end{align*}
Making use of the well-known formula
	$[X, YZ] = [X, Y]Z + Y[X,Z]$,
we see that
	\begin{align*}
	&\left[ \Yo, \left( X_{\sigma(j)}  \ldots X_{\sigma(n-1)} \right)^{\bbb} \right]\\
	=\,& \sum_{i = j}^{n-1} \left( X_{\sigma(j)}  \ldots X_{\sigma(i-1)} \right)^{\bbb} \left[ \Yo, X_{\sigma(i)}^{\bbb} \right] \left( X_{\sigma(i+1)}  \ldots X_{\sigma(n-1)} \right)^{\bbb} \text{.}
	\end{align*}
	Since
	\begin{align*}
	\left[ \Yo, X_{\sigma(i)}^{\bbb} \right]= X_{\sigma(i)}^{\bbb, \sigma(i)} \text{,}
	\end{align*}
	we have
	\begin{align*}
	\left[ \Yo, \left( X_{\sigma(j)}  \ldots X_{\sigma(n-1)} \right)^{\bbb} \right] = \sum_{i = j}^{n-1} \left( X_{\sigma(j)} \ldots X_{\sigma(n-1)} \right)^{\bbb, \sigma(i)} \text{.}
	\end{align*}
	Thus,
	\begin{align*}
	g&= \sum_{\sigma \in S_{n-1}} \sum_{\bbb \in \BBB{r-k}} \sum_{j=1}^{n-1} \lama{\sigma}{\bbbi}{j} \left( X_{\sigma(1)} \ldots X_{\sigma(j-1)} \right)^{\bbb} \sum_{i = j}^{n-1} \left( X_{\sigma(j)} \ldots X_{\sigma(n-1)} \right)^{\bbb,\sigma(i)} \\
	&= \sum_{\sigma \in S_{n-1}} \sum_{\bbb \in \BBB{r-k}} \sum_{j=1}^{n-1} \sum_{i = j}^{n-1} \lama{\sigma}{\bbbi}{j} \left( X_{\sigma(1)} \ldots X_{\sigma(n-1)} \right)^{\bbb,\sigma(i)}\text{.}
	\end{align*}
By changing the order of summation, we obtain
	\begin{align*}
	g = \sum_{\sigma \in S_{n-1}} \sum_{\bbb \in \BBB{r-k}} \sum_{i = 1}^{n-1} \left( \sum_{j=1}^{i} \lama{\sigma}{\bbbi}{j} \right) \left( X_{\sigma(1)} \ldots X_{\sigma(n-1)} \right)^{\bbb, \sigma(i)} \text{.}
	\end{align*}
	As we can see, $g$ is an admissible partially commutative polynomial (of type two) in $n-1$ noncommuting variables.
	Since the image of $g$ is a subset of $f(A; \y)$, it cannot contain $\Ao$. Hence, by our assumption, $g$ is zero.
	By Proposition \ref{lin neodvisnost}, we have $\sum_{j=1}^i \lama{\sigma}{\bbbi}{j} = 0$ for every $\sigma \in S_{n-1}$, every $\bbb \in \BBB{r-k}$, and every $ i = 1,\ldots, n$ (the $i=n$ case follows from the first part of the proof). This obviously implies that $\lama{\sigma}{\bbb'}{j} = 0$ for every  $\sigma \in S_{n-1}$, every $\bbbi \in \BBBi{r}$, and every $j = 1,\ldots, n$. However, this is in  contradiction with $k \in K$.
\end{proof}

It remains to consider admissible polynomials of type two. The main idea of our approach is to use Lemma \ref{korak} by applying the homomorphism $\h{k}$. Since the image of $\h{k}$ of an admissible polynomial is not necessarily  admissible,
our goal is to write
$\h{k} \left( \left( X_{\sigma(1)} \ldots X_{\sigma(n)} \right)^{\bbb, \sigma(i)}\right)$
as a linear combination of the polynomials
$\left( X_{\sigma(1)}\ldots X_{\sigma(n)}\right)^{\bbb+\ccc} \Y^{t}$
for some sequences $\ccc$ and nonnegative integers $t$.
The following lemmas are devoted to this purpose. We remark that our calculations are based on
\begin{align} \label{c}
\Y X_i^\bbb = [\Y, X_i^\bbb] + X_i^\bbb \Y = X_i^{\bbb + \e{i}} + X_i^\bbb \Y \text{.}
\end{align}

\begin{lemma} \label{lema43}
	For any $k \in \N$, we have
	\begin{align*} \label{vsota v lemi}
	\h{k} (X_i^{\bbb, i}) = \sum_{s = 1}^k {k \choose s} X_i^{\bbb + s \e{i}} \Y^{k-s} \text{.}
	\end{align*}
\end{lemma}

\begin{proof}
	We proceed by induction on $k$. Since
	$\h{1}(X_i^{\bbb, i}) =  X_i^{\bbb+\e{i}}$,
	the lemma is obviously true for $k=1$, so assume that $k > 1$ and that the lemma is true for $k-1$. We have
	$$ \h{k}(X_i^{\bbb, i})= [\Y^k, X_i^\bbb] = \Y[\Y^{k-1}, X_i^\bbb] + [\Y, X_i^\bbb] \Y^{k-1} \text{.}$$
	As $[\Y^{k-1}, X_i^\bbb] = \h{k-1}(X_i^{\bbb, i})$, it follows from the induction hypothesis  that
	\begin{align*}
	  \h{k}(X_i^{\bbb, i}) = \Y\sum_{s=1}^{k-1} {k-1 \choose s} X_i^{\bbb + s\e{i}} \Y^{k-1-s} + X_i^{\bbb + \e{i}} \Y^{k-1}. 
		\end{align*}
		Using \eqref{c}, it follows that
	\begin{align*}
	  \h{k}(X_i^{\bbb, i}) =\,&\sum_{s=1}^{k-1} {k-1 \choose s} X_i^{\bbb + (s+1)\e{i}} \Y^{k-1-s} + \sum_{s=1}^{k-1} {k-1 \choose s} X_i^{\bbb + s\e{i}} \Y^{k-s} \\
	  &+ X_i^{\bbb + \e{i}} \Y^{k-1} \text{.}
	\end{align*}
	By changing the index of summation in the first sum and splitting the second sum, we obtain
	\begin{align*}
	\h{k}(X_i^{\bbb, i})
	=\,&\sum_{s=2}^{k} {k-1 \choose s-1} X_i^{\bbb + s\e{i}} \Y^{k-s} + \sum_{s=2}^{k-1} {k-1 \choose s} X_i^{\bbb + s\e{i}} \Y^{k-s} \\
	&+ k X_i^{\bbb + \e{i}} \Y^{k-1} \text{,}
	\end{align*}
and hence
	\begin{align*}
	\h{k}(X_i^{\bbb, i})
	=\,&X_i^{\bbb + k \e{i}} + \sum_{s=2}^{k-1} \left( {k-1 \choose s-1} + {k-1 \choose s} \right) X_i^{\bbb + s\e{i}} \Y^{k-s} \\
	&+ k X_i^{\bbb + \e{i}} \Y^{k-1} \text{.}
	\end{align*}
	Using 
	$${k-1 \choose s-1} + {k-1 \choose s} = {k \choose s}$$
	we obtain the conclusion of the lemma.
	\end{proof}

We  introduce the set
\[
\DDD{t}{i} = \left\{ \ddd = (\ddd_1, \ldots, \ddd_n) \in \No^n \,\, \middle\vert \,\, \ddd_1 = \ldots = \ddd_{i-1} = 0,~ \sum_{j=1}^n \ddd_{j} = t\right\} \text{.}
\]
Note that $\DDD{t}{i}$ is in bijective correspondence with $\sqcup_{s=0}^t \DDD{t-s}{i+1}$, the disjoint union of the sets $\DDD{t-s}{i+1}$,
via $\ddd = s \e{i} + \ddd'$.
We  will use this in the proof of the following lemma.

\begin{lemma} \label{lema44}
	For any $k \in \N$ and $s \in \{ 1, \ldots, k \}$, we have
	\begin{align*}
	\Y^{k-s} \left( X_{i+1} \ldots X_{n} \right)^{\bbb} = \sum_{t = s}^{k} {k - s \choose t - s} \sum_{\ddd \in \DDD{t - s}{i+1}} { t - s \choose \ddd_{i+1} ,\ldots, \ddd_n} \left( X_{i+1}\ldots X_{n}\right)^{\bbb+\ddd} \Y^{k - t} \text{.}
	\end{align*}
\end{lemma}

\begin{proof}
	We proceed by induction on $n-i$. For $n-i = 1$,  we see by repeated use of \eqref{c} that
	\begin{align} \label{prejsnja lema}
	  \Y^{k-s} X_{n}^{\bbb} = \sum_{t=0}^{k-s} {k -s \choose t} X_n^{b + t \e{n}} \Y^{k-s-t} \text{.}
	\end{align}
	By changing the index  of summation and considering $\DDD{t-s}{n} = \{ (t - s) \e{n}\}$, we see that this is further  equal to
	\[
	  \Y^{k-s} X_{n}^{\bbb} = \sum_{t = s}^{k} {k - s \choose t - s} \sum_{\ddd \in \DDD{t - s}{n}} { t - s \choose \ddd_n} X_{n}^{\bbb+\ddd} \Y^{k - t} \text{,}
	\]
	which proves the base case.
	
	Now let $n-i > 1$ and assume that the lemma is true for all numbers smaller than $n-i$. Use \eqref{prejsnja lema} and the induction hypothesis to obtain
	\begin{align*}
	&\Y^{k-s} \left( X_{i+1} \ldots X_{n} \right)^{\bbb}\\
	=\,& \sum_{s' = 0}^{k-s} {k - s \choose s'} X_{i+1}^{\bbb + s' \e{i+1}} \Y^{k-s-s'}\left( X_{i+2} \ldots X_{n} \right)^{\bbb} \\
	 =\,& \sum_{s' = 0}^{k-s} {k - s \choose s'} X_{i+1}^{\bbb + s' \e{i+1}} \\
	  & \cdot \sum_{t = s + s'}^{k} {k - s - s' \choose t - s - s'} \sum_{\ddd \in \DDD{t - s -s'}{i+2}} { t - s - s' \choose \ddd_{i+2} ,\ldots, \ddd_n} \left( X_{i+2}\ldots X_{n}\right)^{\bbb+\ddd} \Y^{k - t}
	  \text{.}
	\end{align*}
	Since
	\begin{align}\label{bin}
		{k - s \choose s'} {k - s - s' \choose t - s - s'} { t - s - s' \choose \ddd_{i+2} ,\ldots, \ddd_n} = {k - s \choose t - s} {t - s \choose s', \ddd_{i+2} ,\ldots, \ddd_n} \text{,}
	\end{align}
	we have
	\begin{align*}
		&\Y^{k-s} \left( X_{i+1} \ldots X_{n} \right)^{\bbb}\\
		=\,& \sum_{s' = 0}^{k-s} \sum_{t = s + s'}^{k} {k - s \choose t-s} \sum_{\ddd \in \DDD{t - s -s'}{i+2}} { t - s \choose s', \ddd_{i+2} ,\ldots, \ddd_n} \left( X_{i+1}\ldots X_{n}\right)^{\bbb+s'\e{i+1}+\ddd} \Y^{k - t}
		\text{.}
	\end{align*}
	Change the order of summation and use the aforementioned bijective correspondence between $\DDD{t-s}{i+1}$ and $\sqcup_{s'=0}^{t-s} \DDD{t-s-s'}{i+2}$ to obtain
	\begin{align*}
		&\Y^{k-s} \left( X_{i+1} \ldots X_{n} \right)^{\bbb}\\
		=\,& \sum_{t = s}^{k} {k - s \choose t-s}\sum_{s' = 0}^{t-s} \sum_{\ddd \in \DDD{t - s -s'}{i+2}} { t - s \choose s', \ddd_{i+2} ,\ldots, \ddd_n} \left( X_{i+1}\ldots X_{n}\right)^{\bbb+s'\e{i+1}+\ddd} \Y^{k - t} \\
		=\,& \sum_{t = s}^{k} {k - s \choose t-s} \sum_{\ddd \in \DDD{t - s}{i+1}} { t - s \choose \ddd_{i+1}, \ddd_{i+2} ,\ldots, \ddd_n} \left( X_{i+1}\ldots X_{n}\right)^{\bbb+\ddd} \Y^{k - t}
		\text{,}
	\end{align*}
	which concludes the induction step.
\end{proof}

Recall from Section \ref{s3} that
\[
\CCC{\sigma}{t}{i} = \left\{ \ccc = (\ccc_1, \ldots, \ccc_n) \in \No^n \,\, \middle\vert \,\, \ccc_{\sigma(1)} = \ldots = \ccc_{\sigma(i-1)} = 0,~\ccc_{\sigma(i)} \geq 1,~ \sum_{j=1}^n \ccc_{j} = t\right\}. 
\]
The set $\CCC{{\rm id}}{t}{i}$ is in bijective correspondence with $\sqcup_{s=1}^t \DDD{t-s}{i+1}$
via $\ccc = s \e{i} + \ddd$.	
Recall also that
	\[
	\skal{\sigma}{\ccc}{i} = {\sum_{j=1}^n \ccc_j \choose \ccc_{\sigma(i)}, \ldots, \ccc_{\sigma(n)}} 
	\]
and set 
	\[
	\pol{\sigma}{\bbb}{i}{t} = \sum_{\ccc \in \CCC{\sigma}{t}{i}} \skal{\sigma}{\ccc}{i}  \left( X_{\sigma(1)} \ldots X_{\sigma(n)} \right)^{\bbb+ \ccc}.
	\]

\begin{lemma} \label{normalna oblika}
	For any $r \in \No$, $\sigma \in S_n$, $\bbb \in \BBB{r}$, $i \in \{1, \ldots,  n\}$, and $k \in \N$, we have
	\begin{align*}
	\h{k} \left( \left( X_{\sigma(1)} \ldots X_{\sigma(n)} \right)^{\bbb, \sigma(i)}\right) = \sum_{t=1}^k {k \choose t} \pol{\sigma}{\bbb}{i}{t} \Y^{k-t}\text{.}
	\end{align*}
\end{lemma}

\begin{proof}
	We will prove the lemma only for the case where $\sigma={\rm id}$
	(if   $\sigma\ne {\rm id}$, just permute the indices in the  formulas that follow). First, we write
	\[
	\h{k} \left( \left( X_{1} \ldots X_{n} \right)^{\bbb, i}\right)
	= \left( X_{1} \ldots X_{i-1} \right)^{\bbb} \h{k} \left( X_{i}^{\bbb, i}\right) \left( X_{i+1} \ldots X_{n} \right)^{\bbb}
	\]
	and apply Lemma \ref{lema43} to obtain
	\[
	\h{k} \left( \left( X_{1} \ldots X_{n} \right)^{\bbb, i}\right)
	= \left( X_{1} \ldots X_{i-1} \right)^{\bbb} \sum_{s = 1}^k {k \choose s} X_i^{\bbb + s \e{i}} \Y^{k-s} \left( X_{i+1} \ldots X_{n} \right)^{\bbb} \text{.}
	\]
	By Lemma \ref{lema44}, 
	\begin{align*}
		\h{k} \left( \left( X_{1} \ldots X_{n} \right)^{\bbb, i}\right)
		=\,& \left( X_{1} \ldots X_{i-1} \right)^{\bbb} \sum_{s = 1}^k {k \choose s} X_i^{\bbb + s \e{i}} \\
		&\cdot\sum_{t = s}^{k} {k - s \choose t - s} \sum_{\ddd \in \DDD{t - s}{i+1}} { t - s \choose \ddd_{i+1} ,\ldots, \ddd_n} \left( X_{i+1}\ldots X_{n}\right)^{\bbb+\ddd} \Y^{k-t} \text{.}
	\end{align*}
	Applying \eqref{bin},
		we have
	\begin{align*}
		&\h{k} \left( \left( X_{1} \ldots X_{n} \right)^{\bbb, i}\right) \\
		=\,& \sum_{s = 1}^k \sum_{t = s}^{k} {k \choose t} \sum_{\ddd \in \DDD{t - s}{i+1}} { t \choose s, \ddd_{i+1} ,\ldots, \ddd_n} \left( X_{1}\ldots X_{n}\right)^{\bbb+s\e{i}+\ddd} \Y^{k-t} \text{.}
	\end{align*}
	Change the order of summation and use the aforementioned bijective correspondence between $\CCC{{\rm id}}{t}{i}$ and
	$\sqcup_{s=1}^t \DDD{t-s}{i+1}$ to obtain
	\begin{align*}
		&\h{k} \left( \left( X_{1} \ldots X_{n} \right)^{\bbb, i}\right) \\
		=\,& \sum_{t = 1}^k {k \choose t} \sum_{s = 1}^{t} \sum_{\ddd \in \DDD{t - s}{i+1}} { t \choose s, \ddd_{i+1} ,\ldots, \ddd_n} \left( X_{1}\ldots X_{n}\right)^{\bbb+s\e{i}+\ddd} \Y^{k-t} \\
		=\,& \sum_{t = 1}^k {k \choose t} \sum_{\ccc \in \CCC{{\rm id}}{t}{i}} { t \choose \ccc_i, \ccc_{i+1} ,\ldots, \ccc_n} \left( X_{1}\ldots X_{n}\right)^{\bbb+\ccc} \Y^{k-t}\text{,}
	\end{align*}
	which concludes this proof.
	\end{proof}

\begin{theorem} \label{mt} Let $A$ be a unital algebra over a field $F$. Suppose $A$ contains an element $\y$ such that $1\in \Ao=\bigcap_{k\in\N}[\y,A]_k$. Then 
	$\Ao\subseteq f(A;\y)$ for every nonzero admissible partially commutative polynomial $f$.
\end{theorem}

\begin{proof}
	We proceed by induction on $n$, i.e., the number of noncommuting variables $X_1, \ldots, X_n$ involved in $f$. The case where $n=1$ was considered in Lemma \ref{baza}.
	
		Let $n >1$. Assume the theorem is true for all nonzero admissible partially commutative polynomials in $n-1$ noncommuting variables.
		In light of Lemma \ref{korak}, it suffices to prove that the theorem is true for every nonzero admissible partially commutative polynomial  of type two in
		$n$ noncommuting variables, i.e., a polynomial of the form
	\begin{align*}
	f = \sum_{\sigma \in S_n} \sum_{\bbb \in \BBB{r}} \sum_{i=1}^n \lam{\sigma}{\bbb}{i} \left( X_{\sigma(1)} \ldots X_{\sigma(n)} \right)^{\bbb, \sigma(i)}
	\end{align*}
	for some $\lam{\sigma}{\bbb}{i} \in F$, not all zero, and some $r \in \No$.

	Suppose that the image of $f$ does not contain $\Ao$. We claim that
	\begin{align} \label{polinomska enacba}
	\sum_{\sigma \in S_n} \sum_{\bbb \in \BBB{r}} \sum_{i=1}^n \lam{\sigma}{\bbb}{i} \pol{\sigma}{\bbb}{i}{k} = 0\text{,}
	\end{align}
	where $\pol{\sigma}{\bbb}{i}{k}$ are as in the preceding lemma, holds for every $k \in \N$.
	Let us prove this by induction on $k$. Since the image of $f$ does not contain $\Ao$, the same applies to the image of $\hi(f)$.
	We have
	\begin{align*}
	\hi(f) = \sum_{\sigma \in S_n} \sum_{\bbb \in \BBB{r}} \sum_{i=1}^n \lam{\sigma}{\bbb}{i} \hi \left( \left( X_{\sigma(1)} \ldots X_{\sigma(n)} \right)^{\bbb, \sigma(i)} \right) \text{,}
	\end{align*}
	which is, by Lemma \ref{normalna oblika},  equal to
	\begin{align*}
	\hi(f) = \sum_{\sigma \in S_n} \sum_{\bbb \in \BBB{r}} \sum_{i=1}^n \lam{\sigma}{\bbb}{i} \pol{\sigma}{\bbb}{i}{1} \text{.}
	\end{align*}
	Since $\pol{\sigma}{\bbb}{i}{1}=\left( X_{\sigma(1)} \ldots X_{\sigma(n)} \right)^{\bbb + \ccc}$ and, for $\bbb \in \BBB{r}$ and $\ccc \in \CCC{\sigma}{1}{i} = \{ \e{\sigma(i)}\}$, the sequence $\bbb + \ccc$ is an element of $\BBB{r+1}$, $\hi(f)$ is an admissible partially commutative polynomial of type one in $n$ noncommuting variables. Since the image of $\hi(f)$ does not contain $\Ao$,  $\hi(f)=0$ by Lemma \ref{korak}. This proves \eqref{polinomska enacba} for $k=1$.

	Now, let $k > 1 $ and assume that \eqref{polinomska enacba} holds for all positive integers smaller than $k$. Since the image of $f$ does not contain $\Ao$, the same applies to the image of $\h{k}(f)$.
	We have
	\begin{align*}
	\h{k}(f) = \sum_{\sigma \in S_n} \sum_{\bbb \in \BBB{r}} \sum_{i=1}^n \lam{\sigma}{\bbb}{i} \h{k}\left( \left( X_{\sigma(1)} \ldots X_{\sigma(n)} \right)^{\bbb, \sigma(i)} \right) \text{,}
	\end{align*}
	which is, by Lemma \ref{normalna oblika},  equal to
	\begin{align*}
	\h{k}(f) &= \sum_{\sigma \in S_n} \sum_{\bbb \in \BBB{r}} \sum_{i=1}^n \lam{\sigma}{\bbb}{i} \sum_{t=1}^k {k \choose t} \pol{\sigma}{\bbb}{i}{t} \Y^{k-t} \\
	&= \sum_{t=1}^k {k \choose t} \left( \sum_{\sigma \in S_n} \sum_{\bbb \in \BBB{r}} \sum_{i=1}^n \lam{\sigma}{\bbb}{i} \pol{\sigma}{\bbb}{i}{t} \right) \Y^{k-t} \text{.}
	\end{align*}
The expression in parenthesis is equal to the left-hand side of  \eqref{polinomska enacba}. By the induction hypothesis,  we have
	\begin{align*}
	\h{k}(f) &={k \choose k} \left( \sum_{\sigma \in S_n} \sum_{\bbb \in \BBB{r}} \sum_{i=1}^n \lam{\sigma}{\bbb}{i} \pol{\sigma}{\bbb}{i}{k} \right) \Y^{k-k} \\
	&=\sum_{\sigma \in S_n} \sum_{\bbb \in \BBB{r}} \sum_{i=1}^n \lam{\sigma}{\bbb}{i} \pol{\sigma}{\bbb}{i}{k}\text{.}
	\end{align*}
	Since $\pol{\sigma}{\bbb}{i}{k}$ is a linear combination of the polynomials $\left( X_{\sigma(1)} \ldots X_{\sigma(n)} \right)^{\bbb + \ccc}$ and, for $\bbb \in \BBB{r}$ and $\ccc \in \CCC{\sigma}{k}{i}$, the sequence $\bbb + \ccc$ is an element of $\BBB{r+k}$, $\h{k}(f)$ is an admissible partially commutative polynomial of type one in $n$ noncommuting variables.
	Since the image of $\h{k}(f)$ does not contain $\Ao$,  $\h{k}(f)=0$ by  Lemma \ref{korak}.
	This concludes our induction step. Thus, we have proven \eqref{polinomska enacba} for every $k \in \N$.

Using the definition of $\pol{\sigma}{\bbb}{i}{k}$ in \eqref{polinomska enacba} we obtain
	\begin{align*}
	\sum_{\sigma \in S_n} \sum_{\bbb \in \BBB{r}} \sum_{i=1}^n \sum_{\ccc \in \CCC{\sigma}{k}{i}} \lam{\sigma}{\bbb}{i} \skal{\sigma}{\ccc}{i} \left( X_{\sigma(1)} \ldots X_{\sigma(n)} \right)^{\bbb + \ccc} = 0 \text{,}
	\end{align*}
	that is,
	\begin{align} \label{lepa enacba}
	\sum_{\sigma \in S_n} \sum_{i=1}^n \sum_{\ccc \in \CCC{\sigma}{k}{i}} \skal{\sigma}{\ccc}{i}\sum_{\bbb \in \BBB{r}} \lam{\sigma}{\bbb}{i} \left( X_{\sigma(1)} \ldots X_{\sigma(n)} \right)^{\bbb + \ccc} = 0 \text{.}
	\end{align}
By changing the index of summation to $\bbb' =  \bbb + \ccc$, we have 
	\[
	\sum_{\bbb \in \BBB{r}} \lam{\sigma}{\bbb}{i} \left( X_{\sigma(1)} \ldots X_{\sigma(n)} \right)^{\bbb+ \ccc} =\sum_{\bbb' \in \BBB{r} + \ccc} \lam{\sigma}{\bbb' - \ccc}{i} \left( X_{\sigma(1)} \ldots X_{\sigma(n)} \right)^{\bbb'} \text{.}
	\]
	Since $\BBB{r}+ \ccc$ is a subset of $\BBB{r+k}$, we further have
	\begin{align} \label{lep zapis} 	\sum_{\bbb \in \BBB{r}} \lam{\sigma}{\bbb}{i} \left( X_{\sigma(1)} \ldots X_{\sigma(n)} \right)^{\bbb+ \ccc} =
	\sum_{\bbb' \in \BBB{r+k}} \mm{\sigma}{\bbb' - \ccc}{i} \left( X_{\sigma(1)} \ldots X_{\sigma(n)} \right)^{\bbb'} \text{,}
	\end{align}
	where
	\[
	\mm{\sigma}{\bbb}{i} = \begin{cases}\lam{\sigma}{\bbb}{i},~ \bbb \in \BBB{r} \\ 0,~\bbb \not\in \BBB{r} \end{cases} \text{.}
	\]
Using \eqref{lep zapis} in \eqref{lepa enacba} we get
	\begin{align*}
	\sum_{\sigma \in S_n} \sum_{i=1}^n \sum_{\ccc \in \CCC{\sigma}{k}{i}} \skal{\sigma}{\ccc}{i}\sum_{\bbb' \in \BBB{r+k}} \mm{\sigma}{\bbb' - \ccc}{i} \left( X_{\sigma(1)} \ldots X_{\sigma(n)} \right)^{\bbb'} = 0 \text{,}
	\end{align*}
	or, equivalently,
	\begin{align*}
	\sum_{\sigma \in S_n} \sum_{\bbb' \in \BBB{r+k}} \left( \sum_{i=1}^n \sum_{\ccc \in \CCC{\sigma}{k}{i}} \skal{\sigma}{\ccc}{i} \mm{\sigma}{\bbb' - \ccc}{i} \right) \left( X_{\sigma(1)} \ldots X_{\sigma(n)} \right)^{\bbb'} = 0 \text{.}
	\end{align*}
	By Proposition \ref{lin neodvisnost}, 
	\[
	\sum_{i=1}^n \sum_{\ccc \in \CCC{\sigma}{k}{i}} \skal{\sigma}{\ccc}{i} \mm{\sigma}{\bbb' - \ccc}{i} = 0
	\]
	for every $\sigma \in S_n$, every $k \in \N$, and every $\bbb' \in \BBB{r+k}$.
Now, Proposition \ref{grozna trditev} tells us that each $\mm{\sigma}{\bbb'- \ccc}{i}$ is zero, and consequently, each $\lam{\sigma}{\bbb}{i}$ is zero. We have thus arrived at a contradiction to the assumption that $f\ne 0$.
\end{proof}

Finally, we point out that since multilinear noncommutative polynomials are special examples of admissible partially commutative polynomials (see Remark \ref{opomba21}), Theorem \ref{t} follows immediately from Theorem \ref{mt}.

\section*{Acknowledgment}

The author would like to thank his supervisor Matej Brešar for his constant guidance, support, and tireless efforts in improving this paper.

\end{document}